\documentclass[reqno,centertags,a4paper,12pt]{amsart}
\usepackage{amssymb}
\usepackage{amsthm}
\usepackage{eucal}
\newcommand{\R}{\mathbb R}

\numberwithin{equation}{section}

\newtheorem{theorem}{Theorem}[section]
\newtheorem{proposition}[theorem]{Proposition}
\newtheorem{remark}[theorem]{Remark}
\newtheorem{lemma}[theorem]{Lemma}

\begin{document}
\title[Sharp local well-posedness]
{Sharp local well-posedness of KdV type equations with dissipative perturbations}

\author[X. Carvajal, M. Panthee]
{Xavier Carvajal, Mahendra Panthee} 

\address{Xavier Carvajal \newline
Instituto de Matem\'atica - UFRJ
Av. Hor\'acio Macedo, Centro de Tecnologia
 Cidade Universit\'aria, Ilha do Fund\~ao,
21941-972 Rio de Janeiro,  RJ, Brasil}
\email{carvajal@im.ufrj.br}

\address{Mahendra Panthee \newline
IMECC, UNICAMP
13083-859, Campinas, SP, Brazil}
\email{mpanthee@ime.unicamp.br}

\thanks{This work was partially supported by FAPESP, Brazil.}

\begin{abstract}
 In this work, we study the initial value problems associated to some linear
 perturbations of KdV equations. Our focus is in the well-posedness issues
 for initial data given in the $L^2$-based Sobolev spaces.
 We derive bilinear estimate in a space with weight in the time variable and obtain sharp local well-posedness results.
\end{abstract}

\maketitle

\allowdisplaybreaks

\section{Introduction}

In this article, continuing our earlier work \cite{XC-MP2}, we consider the following initial value problems (IVPs)
\begin{equation}\label{eq:hs}
 \begin{gathered}
  v_t+v_{xxx}+\eta Lv+(v^{2})_x=0, \quad x \in \mathbb{R}, \; t\geq 0,\\
     v(x,0)=v_0(x),
 \end{gathered}
\end{equation}
and
\begin{equation}\label{eq:hs-1}
 \begin{gathered}
  u_t+u_{xxx}+\eta Lu+(u_x)^{2}=0, \quad x \in \mathbb{R}, \,t\geq 0,\\
  u(x,0)=u_0(x),
 \end{gathered}
\end{equation}
where $\eta>0$ is a constant; $u=u(x, t)$, $v=v(x,t)$ are real valued functions and the 
linear operator $L$ is defined via the Fourier transform by 
$\widehat{Lf}(\xi)=-\Phi(\xi)\hat{f}(\xi)$.

The Fourier symbol $\Phi(\xi)$ is of the form
\begin{align}\label{phi}
\Phi(\xi)=-|\xi|^p + \Phi_1(\xi),
\end{align}
where $p\in \R^+$ and $|\Phi_1(\xi)|\leq C(1+|\xi|^q)$ with $0\leq q<p$. We note that the symbol $\Phi(\xi)$ is a real valued function which is bounded above; i.e., there is a constant $C$ such  that $\Phi(\xi) < C$ (see Lemma \ref{lem2.2} below). 
 In our earlier work \cite{XC-MP2}, we considered a particular case of $\Phi(\xi)$ in the form
\begin{align}\label{phi-9}
\tilde{\Phi}(\xi)=\sum_{j=0}^{n}\sum_{i=0}^{2m}c_{i,j}\xi^i |\xi|^j, \quad c_{i,j} \in \mathbb{R},
\; c_{2m,n}=-1,
\end{align}
with $p:=2m+n$.

We observe that,  if $u$ is a solution of \eqref{eq:hs-1} then $v=u_x$ is a solution 
of \eqref{eq:hs} with initial data $v_0 = (u_0)_x$. That is why \eqref{eq:hs} is called  
the derivative equation of \eqref{eq:hs-1}.

In this work, we are interested in investigating the  well-posedness results 
to the IVPs \eqref{eq:hs-1} and \eqref{eq:hs} for given data in the low
regularity  Sobolev spaces $H^s(\mathbb{R})$. Recall that, for $s\in \mathbb{R}$, 
the $L^2$-based Sobolev spaces $H^s(\mathbb{R})$ are defined by
$$
H^s(\mathbb{R}) := \{f\in \mathcal{S}'(\mathbb{R}) : \|f\|_{H^s} < \infty\},
$$
where 
$$
\|f\|_{H^s} := \Big(\int_{\mathbb{R}} (1+|\xi|^2)^s|\hat f(\xi)|^2d\xi\Big)^{1/2},
$$
and $\hat f(\xi)$ is the usual Fourier transform given by
$$
\hat f(\xi) \equiv \mathcal{F}(f)(\xi) 
:= \frac 1{\sqrt{2\pi}}\int_{\mathbb{R}}e^ {-ix\xi} f(x)\, dx.
$$
The factor $\frac1{\sqrt{2\pi}}$ in the definition of the Fourier transform  
does not alter our analysis, so we will omit  it.

The notion of well-posedness we use is the standard one. We say that an IVP 
for given data in a Banach space $X$ is locally well-posed, if there exists 
a certain time interval $[0, T]$ and a unique solution depending continuously
upon the initial data and the solution satisfies the persistence property; 
i.e., the solution describes a continuous curve in $X$ in the time interval  $[0, T]$.
If the above properties are true for any time interval, we say that the IVP is globally 
well-posed. If any one of the above properties fails to hold,  we say that the problem is ill-posed.

The notion of ill-posedness used in this work is abit different from the standard one. If one uses contraction mapping principle, the application data-solution turns out to be always smooth (see for example \cite{kpv1:kpv1}). There are many works in the literature (see \cite{B97}, \cite{MRY}, \cite{MR}, \cite{Tz} and references therein) where the notion of well-posedness has been strengthened by requiring the smoothness of the mapping data-solution. In this work too, we follow this notion of ``well-posedness'' and say that the IVP is ``ill-posed'' if the mapping data-solution fails to be smooth.

The function space in which we work is  motivated from the one introduced in \cite{Dix}, where the author proved the sharp local well-posedness for the IVP associated to the Burgers' equation by showing that the local well-posedness holds for data in $H^s(\R)$, if $-\frac12 <s\leq 0$ and uniqueness fails if $s<-\frac12$. The new ingredient  in \cite{Dix} was the use of the function space with time dependent weight. A natural question is: whether such result still holds true if one considers higher order dissipative equation. And what happens if one uses dispersive term $v_{xxx}$ as in \eqref{eq:hs}? Recently, an analysis in direction is carried out in \cite{Amin} where the author considered Ostrovsky-Stepanyams-Tsimring equation that is a particular case of \eqref{eq:hs} containing dissipative term with leading order $3$ and obtained a sharp local well-posedness result for data in $H^s(\R)$, $-\frac32<s\leq 0$.  In this work we introduce  suitable function spaces with time dependent weight and   prove sharp  local well-posedness results for the IVPs \eqref{eq:hs} and \eqref{eq:hs-1} when the order of the leading dissipative term is bigger than $3$. More precisely, for $p>0$ and $t\in [0, T]$ with $0\leq T\leq 1$, these spaces are defined with weight in time variable via the norms
\begin{equation}\label{sp.1}
\|f\|_{X_T^s} :=\sup_{t\in[0,T]}\Big\{\|f(t)\|_{H^s}+t^{\frac{|s|}{p}}\|f(t)\|_{L^2}\Big\},
\end{equation}
and
\begin{equation}\label{sp.2}
\|f\|_{Y_T^s} :=\sup_{t\in[0,T]}\Big\{\|f(t)\|_{H^s}+t^{\frac{1+|s|}{p}}\|\partial_xf(t)\|_{L^2}\Big\},
\end{equation}
and will be used to prove local  well-posedness for the IVPs \eqref{eq:hs} and \eqref{eq:hs-1} respectively. We use notation $\langle\cdot\rangle = (1+|\cdot|)$.

The first main result of this work  is about the local well-posednness of the IVP \eqref{eq:hs} and reads as follows.
\begin{theorem}\label{teorp-1}
 Let $\eta>0$ be fixed and $\Phi(\xi)$ be given by \eqref{phi} with $p>3$ as the order of the leading term. Then  for any data $v_0 \in H^s(\mathbb{R})$,  $s>-\frac{p}2$ there exist a time $T=T(\|v_0\|_{H^s})$ and a unique solution $v$ to the IVP \eqref{eq:hs} in $C([0, T], H^s(\R))$.  

Moreover, the map $v_0\mapsto v$ is smooth from $H^s(\R)$ to $C([0, T]; H^s(\R))\cap X_T^s$ and $v\in C((0, T]; H^{\infty}(\R))$.
\end{theorem}

The  the second result deals the same for the IVP \eqref{eq:hs-1}, with low regularity data.
\begin{theorem}\label{teorp}
Let $\eta>0$ be fixed and $\Phi(\xi)$ be given by \eqref{phi} with $p>3$ as the order of the leading term. Then for any data $u_0 \in H^s(\mathbb{R})$,  $s>1-\frac{p}2$  there exist a time $T=T(\|v_0\|_{H^s})$ and a unique solution $u$ to  the IVP \eqref{eq:hs-1}   in $C([0, T], H^s(\R))$.  

Moreover, the map $v_0\mapsto v$ is smooth from $H^s(\R)$ to $C([0, T]; H^s(\R))\cap Y_T^s$ and $u\in C((0, T]; H^{\infty}(\R))$.
\end{theorem}

Our next task is to check is the well-posedness results obtained in Theorems \ref{teorp-1} and \ref{teorp} are optimal. To have  an insight in this issue, we analyze it by using scaling argument. As the regulairty requirement for the IVP \eqref{eq:hs-1} is one more than that for the IVP \eqref{eq:hs}, we discuss only the later. Talking heuristically, semilinear evolution equations like viscous Burgers, Korteweg-de Vries (KdV), nonlinear Schr\"odinger (NLS) and wave equations are usually expected to be well-posed for given data with Sobolev regularity up to scaling and ill-posed below scaling. However, this is not always true, as can be seen in the  KdV  case. For $\eta=0$, the IVP \eqref{eq:hs} turns out to be the KdV equation
\begin{equation}\label{eq:hsgKdV}
 \begin{cases}
  v_t+v_{xxx}+(v^{2})_x=0, \quad x \in \mathbb{R}, \; t\geq 0, \\
     v(x,0)=v_0(x),
 \end{cases}
\end{equation}
which satisfies the scaling property. Talking more precisely, if $v(x, t)$ is a solution of the gKdV with initial data $v_0(x)$ then for $\lambda>0$, so is $v^{\lambda}(x, t) =\lambda^{2} v(\lambda x, \lambda^3t)$ with initial data $v^{\lambda}(x,0) =\lambda^{2} v(\lambda x, 0)$. Note that the homogeneous Sobolev norm of the initial data remains invariant if $s+\frac32=0$, which suggests that the scaling Sobolev regularity is $-\frac32$. But, Kenig et al. \cite{KPV01, kpv2:kpv2} proved local well-posedness of the IVP \eqref{eq:hsgKdV} for data in $H^s(\R)$, $s>-\frac34$ is sharp since the
flow-map $u_0 \to u(t)$ is not locally uniformly continuous from $\dot{H}^{s} (\R)$ to $\dot{H}^{s} (\R)$, $s< -\frac34$. This result is far above from the critical index suggested by scaling.

Generally, for dissipative problem the scaling index is better in the sense that one can lower the regularity requirement on the data to get well-posedness. As can be seen in the proofs of the Theorems \ref{teorp-1} and \ref{teorp} (below), our method depends on the leading order of $L$.  If we discard the third order derivative (dispersive part) and consider only the dissipative operator $L$ with the Fourier symbol $|\xi|^p$, with $p>0$ in \eqref{eq:hs}, i.e.,
\begin{equation}\label{dissip-1}
\begin{cases}
v_t +\eta L v +(v^{2})_x =0, \qquad \widehat{Lv}(\xi) =|\xi|^p\widehat{v}(\xi),\\
v(x,0) =v_0(x),
\end{cases}
\end{equation}
it is easy to check that, if $v(x,t)$ solves \eqref{dissip-1} with initial data $v(x, 0)$, then for $\lambda>0$ so does  $v^{\lambda}(x, t) =\lambda^{{p-1}} v(\lambda x, \lambda^pt)$ with initial data $v^{\lambda}(x,0) =\lambda^{p-1} v(\lambda x, 0)$. Note that
\begin{equation}\label{sc-1}
\|v^{\lambda}(0)\|_{\dot{H}^s} =\lambda^{p-\frac32+s}\|v(0)\|_{\dot{H}^s}.
\end{equation}
From \eqref{sc-1} we see that the scaling index for this particular situation is $s_c:= \frac32-p$. Observe that for $p=3$ we get $s_c=-\frac32$ which coincides with the scaling critical regularity of the KdV equation. With this observation we see that the  local well-posedness result  proved in \cite{Amin} is up to the Sobolev regularity given by scaling for $p=3$. However, for $p>3$,  as can be seen in Theorems \ref{teorp-1} and \ref{teorp} the regularity requirement for local well-posedness is higher than $s_c$ (i.e., $s_c<-\frac p2$). Since the regularity requirement for the IVP \eqref{eq:hs-1} is higher than 1 to that for the IVP \eqref{eq:hs}, we see that the scaling index for this is $s_c+1$.

A natural question is, whether one can improve the local results in Theorems \ref{teorp-1} and \ref{teorp}  can be improved up to the regularity given by scaling argument.  The following results provide a negative answer to this question.
\begin{theorem}\label{illposed-1}
Let $p\ge 2$, $s< -\frac p2$, then there does not exist any $T>0$ such that the IVP \eqref{eq:hs} admits a unique local solution defined in the interval $[0, T]$ such that the flow-map 
\begin{equation}\label{flow-map}
v_0\mapsto v,
\end{equation}
is $C^2$-differentiable at the origin from $H^s(\R)$ to $C([0, T]; H^s(\R))$.
\end{theorem}
\begin{theorem}\label{illposed-2}
Let $p\ge 2$, $s< 1-\frac p2$, then there does not exist any $T>0$ such that the IVP \eqref{eq:hs-1} admits a unique local solution defined in the interval $[0, T]$ such that the flow-map 
\begin{equation}\label{flow-map2}
u_0\mapsto u, 
\end{equation}
is $C^2$-differentiable at the origin from $H^s(\R)$ to $C([0, T]; H^s(\R))$.
\end{theorem}

As described earlier, we recall that the contraction mapping argument applied in the proof of the well-posedness theorems shows that the mapping data-solution is always smooth. In the light of this observation, Theorems \ref{illposed-1} and \ref{illposed-2} show that the well-posedness results for the IVP \eqref{eq:hs} and \eqref{eq:hs-1} proved in  Theorems \ref{teorp-1} and \ref{teorp} are sharp in the sense that one cannot employ contraction mapping principle for given data with Sobolev regularity below the one given by these theorems.  As in the usual KdV case, local well-posedness cannot be achieved up to the scaling index (suggested by considering only dissipative term) using contraction mapping argument if $p>3$ in \eqref{phi}.

As it can be seen in the proofs  below, our method in this article holds only for $-\frac p2 <s \le 0$  in Theorem \ref{teorp-1} and for
$1-\frac p2 <s \le 0$ in Theorem \ref{teorp}  considering $p>3$. However, for $s > 0$ (Theorem \ref{teorp-1} and Theorem \ref{teorp}) we already have proved local well-posedness in our earlier work \cite{XC-MP1}.
In fact, in \cite{XC-MP1}  we proved that the IVPs \eqref{eq:hs} and \eqref{eq:hs-1} for given data in $H^s(\R)$ are locally well-posed  whenever $s>-1$ and $s>0$ respectively (see also  \cite{XC-MP} and \cite{XC-MP2}).  To obtain these results we followed the techniques developed by Bourgain \cite{bou:bou}, Kenig, Ponce and Vega \cite{kpv2:kpv2} (see also \cite{tao}) and Dix \cite{Dix}. The
main ingredients in the proof  are estimates
in the integral equation associated to an extended IVP that is defined for all 
$t\in \mathbb{R}$.
The main idea in \cite{XC-MP} and \cite{XC-MP2} is to use the usual Bourgain space associated to  the KdV equation 
instead of that associated to the linear part of the IVPs \eqref{eq:hs} and 
\eqref{eq:hs-1}. For the well-posedness issues in the periodic setting we refer to \cite{xavirica}.

 In what follows, we present some particular examples 
that belong to the class considered in \eqref{eq:hs} and \eqref{eq:hs-1} and 
discuss the known well-posedness results about them.

The first examples belonging to the classes  \eqref{eq:hs} and \eqref{eq:hs-1} are the
Korteweg-de Vries-Burgers (KdV-B) equation
\begin{equation}\label{eqhsB}
   \begin{gathered}
     v_t+v_{xxx}-\eta v_{xx}+(v^{2})_x=0, \quad  x \in \mathbb{R}, \; t\geq  0,\\
     v(x,0)=v_0(x),
   \end{gathered}
\end{equation}
and
\begin{equation}\label{eqhs-1B}
   \begin{gathered}
     u_t+u_{xxx}-\eta u_{xx}+(u_x)^{2}=0, \quad x \in \mathbb{R}, \; t\geq     0,\\
     u(x,0)=u_0(x),
   \end{gathered}
\end{equation}
where $u=u(x,t)$, $v=v(x,t)$ are real-valued functions and $\eta>0$ is a constant,
$\Phi(\xi)=-|\xi|^2$ and $p=2$. Equation \eqref{eqhsB} has been derived as a model for the propagation of weakly nonlinear dispersive
long waves in some physical contexts when dissipative effects occur (see \cite{OS}).
For motivation, we refer to the recent work of Molinet and Ribaud \cite{MR} where the authors proved sharp global well-posedness for $s>-1$  in the framework of the Fourier transform restriction norm spaces introduced by Bourgain \cite{bou:bou}.

Next examples that fit in \eqref{eq:hs} and \eqref{eq:hs-1} are
\begin{equation}\label{eqhs}
   \begin{gathered}
     v_t+v_{xxx}-\eta(\mathcal{H}v_x+\mathcal{H}v_{xxx})+(v^{2})_x=0, \quad  x \in \mathbb{R}, \; t\geq  0,\\
     v(x,0)=v_0(x),
   \end{gathered}
\end{equation}
and
\begin{equation}\label{eqhs-1}
   \begin{gathered}
     u_t+u_{xxx}-\eta(\mathcal{H}u_x+\mathcal{H}u_{xxx})+(u_x)^{2}=0, \quad x \in \mathbb{R}, \; t\geq     0,\\
     u(x,0)=u_0(x),
   \end{gathered}
\end{equation}
respectively, where $\mathcal{H}$ denotes the Hilbert transform
\begin{align*}
\mathcal{H}g(x)=\operatorname{P.V.} \frac{1}{\pi}\int\frac{g(x-\xi)}{\xi}d\xi,
\end{align*}
$u=u(x,t)$, $v=v(x,t)$ are real-valued functions and $\eta>0$ is a constant, $\Phi(\xi)=-|\xi|^3+|\xi|$ and $p=3$.

The equation in \eqref{eqhs}  was derived by Ostrovsky et al
\cite{O:O} to describe the radiational  instability of long
waves in a stratified shear flow.  Recently,  Carvajal and Scialom \cite{Cv-Sc} 
considered the IVP \eqref{eqhs} and proved the
local well-posedness results for given data  in $H^s$, $s \geq 0$. They also obtained an \emph{a priori} estimate for given data in $L^2(\mathbb{R})$
there by proving  global well-posedness result. The earlier well-posedness results 
for \eqref{eqhs} can be found in \cite{pa:ba1}, where for given 
data in $H^s(\mathbb{R})$, local well-posedness when $s>1/2$ and global well-posedness when
$s\geq 1$ have been proved. In \cite{pa:ba1}, IVP \eqref{eqhs-1} 
is also considered to prove global well-posedness for given data in $H^s(\mathbb{R})$, 
$s\geq1$. 

Another two  models that fit in the classes \eqref{eq:hs-1} and \eqref{eq:hs} respectively 
are the Korteweg-de Vries-Kuramoto Sivashinsky (KdV-KS) equation
\begin{equation}\label{1eqhs-1}
   \begin{gathered}
     u_t+u_{xxx}+\eta(u_{xx}+u_{xxxx})+(u_x)^2=0, \quad x \in \mathbb{R}, \; t\geq 0,\\
     u(x,0)=u_0(x),
   \end{gathered}
\end{equation}
 and its derivative equation
\begin{equation}\label{1eqhs}
   \begin{gathered}
     v_t+v_{xxx}+\eta(v_{xx}+v_{xxxx})+vv_x=0, \quad x \in \mathbb{R}, \;t\geq 0,\\
     v(x,0)=v_0(x),
   \end{gathered}
\end{equation}
where $u=u(x,t)$, $v=v(x,t)$ are real-valued functions and $\eta>0$ is a constant, $\Phi(\xi)=-|\xi|^4+|\xi|^2$ and $p=4$.

The KdV-KS equation arises as a model for long waves in a viscous fluid flowing down 
an inclined plane and also describes drift waves in a plasma (see \cite{CKTR, TK}). 
The KdV-KS equation is very interesting in the sense that it combines the dispersive 
characteristics of the Korteweg-de Vries equation and dissipative characteristics of 
the Kuramoto-Sivashinsky equation. Also, it is worth noticing that 
 \eqref{1eqhs} is a particular case of the Benney-Lin equation \cite{B,TK}; i.e.,
\begin{equation}\label{2eqhs}
   \begin{gathered}
    v_t+v_{xxx}+\eta(v_{xx}+v_{xxxx})+\beta v_{xxxxx}+vv_x=0, \quad
 x \in \mathbb{R},\;t\geq 0,\\
  v(x,0)=v_0(x),
 \end{gathered}
\end{equation}
when $\beta=0$.

The IVPs  \eqref{1eqhs-1} and \eqref{1eqhs} were studied by Biagioni, Bona, 
Iorio and Scialom \cite{BBIS}. The authors in \cite{BBIS} proved that the 
IVPs \eqref{1eqhs-1} and \eqref{1eqhs} are  locally well-posed for given data 
in $H^s$, $s\geq 1$ with $\eta >0$. They also constructed appropriate \emph{a priori}
estimates and used them to prove global well-posedness too.
The limiting behavior of solutions as the dissipation tends to zero
(i.e., $\eta\to 0$) has also been studied in \cite{BBIS}. The IVP \eqref{2eqhs} associated 
to the Benney-Lin equation is also widely studied in the literature \cite{B, BL, TK}.
Regarding well-posedness issues for the IVP \eqref{2eqhs} the work of Biagioni 
and Linares \cite{BL} is worth mentioning, where they proved global well-posedness  
for given data in $L^2(\mathbb{R})$. For the sharp well-posedness result for the KdV-KS equation we refer to the recent work of Pilod in \cite{Pilod} where the author proved local well-posedness in $H^s(\R)$ for $s>-1$ and ill-posedness for $s<-1$. For recent work on generalized Benjamin-Ono-Burgers equation we refer to \cite{otani} where the author uses Bourgain's space to obtain local well-posedness for data with low Sobolev regularity.
 
Now we consider the IVP associated to the linear parts of  \eqref{eq:hs} 
and \eqref{eq:hs-1},
\begin{equation}\label{eq0.5}
   \begin{gathered}
     w_t+w_{xxx}+\eta Lw=0, \quad x, \; t\geq 0,\\
     w(0)=w_0.
   \end{gathered}
\end{equation}
The solution to \eqref{eq0.5} is given by $w(x,t)=V(t)w_0(x)$
where the  semigroup $V(t)$ is defined as
\begin{equation}\label{gV}
\widehat{V(t)w_0}(\xi)=e^{it\xi^3+\eta t\Phi(\xi)}\widehat{w_0}(\xi).
\end{equation}
In what follows, without loss of generality, we suppose $\eta =1$.

This paper is organized as follows: In Section \ref{sec-2}, we prove 
some preliminary estimates.  Sections \ref{sec-3}  and \ref{sec-5} are dedicated to prove the local well-posednes and ill-posedness results respectively.

\section{Preliminary estimates}\label{sec-2}

This section is devoted to obtain linear and nonlinear estimates that are essential 
in the proof of the main results. We start with following estimate that the Fourier symbol defined in \eqref{phi} satisfies.

\begin{lemma}\label{xav4}
There exists $M>0$ large such that for all $|\xi| \geq M$, one has that
\begin{equation}\label{xavi9}
\Phi(\xi)=-|\xi|^{p}+\Phi_1(\xi) <-1,
\end{equation}
\begin{equation}\label{eq0.6}
\frac{\Phi_1(\xi)}{|\xi|^p}\leq \frac12,
\end{equation}
and 
\begin{equation}\label{eq0.6x34}
|\Phi(\xi)| \ge  \frac{|\xi|^p}2.
\end{equation}
\end{lemma}

\begin{proof}
The inequalities \eqref{xavi9} and \eqref{eq0.6} are  direct consequences of
\begin{equation*}
\lim_{\xi \to \infty} \frac{\Phi_1(\xi) +1}{|\xi|^p}=0 \quad {\textrm{and}}\quad \lim_{\xi \to \infty} \frac{\Phi_1(\xi)}{ |\xi|^p}=0,
\end{equation*}
respectively.

The estimate \eqref{eq0.6x34} follows from \eqref{xavi9} and \eqref{eq0.6}. In fact, for $|\xi|>M$
\begin{equation}
|\Phi(\xi)|= |\xi|^p-\Phi_1(\xi) \ge \dfrac{|\xi|^p}{2},
\end{equation}
and this concludes the proof of the  \eqref{eq0.6x34}. 
\end{proof}

\begin{lemma}\label{lem2.2}
The Fourier symbol $\Phi(\xi)$ given by \eqref{phi} is bounded from above and the following estimate holds true
\begin{equation}\label{eq0.11}
\|e^{t\Phi(\xi)}\|_{L^{\infty}} \leq e^{TC_M}.
\end{equation}
\end{lemma}

\begin{proof}
From Lemma \ref{xav4}, there is $M>1$ large enough such that  for $|\xi|\geq M$ one has $\Phi(\xi) <-1$. Consequently, $e^{t\Phi(\xi)} \leq e^{-t}\leq 1$.  Now for $|\xi| <M$, it is easy to get $\Phi(\xi) <C_M$, so that $e^{t\Phi(\xi)} \leq e^{TC_M} $. Therefore, in any case 
\begin{equation*}
\|e^{t\Phi(\xi)}\|_{L^{\infty}} \leq e^{TC_M}.
\end{equation*}
as required.
\end{proof}

The following result is an elementary fact from calculus.
\begin{lemma}\label{ele-1}
Let $f(t)= t^a e^{tb}$ with $a>0$ and $b<0$, then for all $t\geq 0$ one has 
\begin{equation}\label{eq0.01}
f(t)\leq \left(\frac{a}{|b|}\right)^ae^{-a}.
\end{equation}
\end{lemma}

\begin{lemma}\label{lem-smooth} 
Let $V(t)$ be as defined in \eqref{gV} and $v_0\in H^s(\R)$, then $$ V(\cdot)v_0\in C([0, \infty); H^s(\R))\cap C((0, \infty); H^{\infty}(\R)).$$
\end{lemma}
\begin{proof}
It is sufficient to prove that $V(t)v_0\in H^{s'}(\R)$ for $s'>s$. Now,
\begin{equation}\label{sm0.1}
\begin{split}
\|V(t)v_0\|_{H^{s'}} &=\|\langle\xi\rangle^{s'}e^{it\xi^3+t\Phi(\xi)}\widehat{v_0}(\xi)\|_{L^2}\\
&=\|\langle\xi\rangle^{s}\widehat{v_0}(\xi)\langle\xi\rangle^{s'-s}e^{t\Phi(\xi)}\|_{L^2}\\
&\leq \|\langle\xi\rangle^{s'-s}e^{t\Phi(\xi)}\|_{L^{\infty}}\|v_0\|_{H^s}.
\end{split}
\end{equation}
Let $M\gg1$ be as in Lemma \ref{xav4}, then we have
\begin{equation}\label{sm0.2}
\begin{split}
\|\langle\xi\rangle^{s'-s}e^{t\Phi(\xi)}\|_{L^{\infty}}
&\leq \|\langle\xi\rangle^{s'-s}e^{t\Phi(\xi)}\|_{L^{\infty}(|\xi|\leq M)}
+\|\langle\xi\rangle^{s'-s}e^{t\Phi(\xi)}\|_{L^{\infty}(|\xi|> M)}\\
&\leq  C_M + \|e^{-\frac{|\xi|^p}2t}\langle\xi\rangle^{s'-s}\|_{L^{\infty}} <\infty, \qquad t\in \R^+.
\end{split}
\end{equation}
Continuity follows using dominated convergence theorem.
\end{proof}

\begin{lemma}\label{xav5}
Let  $0<T \leq 1$ and $t\in [0, T]$. Then for all $s\in \R$, we have
\begin{equation}\label{eq0.9}
\|V(t)u_0\|_{X_T^s}\lesssim  e^{C_MT} \|u_0\|_{H^s},
\end{equation}
where the constant $C_M$ depends on $M$ with $M$ as in Lemma \ref{xav4}.
\end{lemma}

\begin{proof}
We start by estimating the first component of the $X_T^s$-norm. We have that
\begin{equation}\label{eq0.10}
\|V(t)u_0\|_{H^s} = \|\langle\xi\rangle^se^{t\Phi(\xi)}\widehat{u_0}(\xi)\|_{L^2} \leq \|e^{t\Phi(\xi)}\|_{L^{\infty}}\|u_0\|_{H^s}.
\end{equation}

 Using \eqref{eq0.11} in \eqref{eq0.10}, we get 
\begin{equation}\label{eq0.12}
\|V(t)u_0\|_{H^s} \leq e^{TC_M}\|u_0\|_{H^s}.
\end{equation}

 Now, we move to estimate the second  component of the $X_T^s$-norm. The case $s\geq 0$ is quite easy, so we consider only the case when $s<0$. Using Plancherel, we have
\begin{equation}\label{eq0.13}
\begin{split}
t^{\frac{|s|}p}\|V(t)u_0\|_{L^2} &= t^{\frac{|s|}p}\|e^{t\Phi(\xi)}\widehat{u_0}\|_{L^2} \\
&= t^{\frac{|s|}p}\|\langle\xi\rangle^{-s}e^{t\Phi(\xi)}\langle\xi\rangle^{s}\widehat{u_0}\|_{L^2}\\ 
&\leq t^{\frac{|s|}p}\|\langle\xi\rangle^{|s|}e^{t\Phi(\xi)}\|_{L^{\infty}}\|u_0\|_{H^s}.
\end{split}
\end{equation}

Since $\langle\xi\rangle^{|s|} \lesssim 1+|\xi|^{|s|}$, from \eqref{eq0.13}, one obtains
\begin{equation}\label{eq0.14}
t^{\frac{|s|}p}\|V(t)u_0\|_{L^2} 
\leq t^{\frac{|s|}p}\Big[\|e^{t\Phi(\xi)}\|_{L^{\infty}}+\||\xi|^{|s|}e^{t\Phi(\xi)}\|_{L^{\infty}}\Big]\|u_0\|_{H^s}.
\end{equation}

From \eqref{eq0.11}, we have $\|e^{t\Phi(\xi)}\|_{L^{\infty}}\leq e^{TC_M}$. To estimate $\||\xi|^{|s|}e^{t\Phi(\xi)}\|_{L^{\infty}}$ we proceed as follows.
\begin{equation}\label{eq0.15}
\||\xi|^{|s|}e^{t\Phi(\xi)}\|_{L^{\infty}}\leq \||\xi|^{|s|}e^{t\Phi(\xi)}\chi_{\{|\xi|\leq M\}}\|_{L^{\infty}}+\||\xi|^{|s|}e^{t\Phi(\xi)}\chi_{\{|\xi|
> M\}}\|_{L^{\infty}}.
\end{equation}

For the low-frequency part, it is easy to get 
\begin{equation}\label{eq0.16}
 \||\xi|^{|s|}e^{t\Phi(\xi)}\chi_{\{|\xi|\leq M\}}\|_{L^{\infty}}\leq M^{|s|}e^{C_MT}.
\end{equation}

Now, we move to estimate the high-frequency part  $\||\xi|^{|s|}e^{t\Phi(\xi)}\chi_{\{|\xi|> M\}}\|_{L^{\infty}}$ in \eqref{eq0.15}. For this, we make use of the time weight in the definition of $X_T^s$-norm and define for $|\xi|> M$, $g(t,\xi):= t^{\frac{|s|}p}|\xi|^{|s|} e^{t\Phi(\xi)}$. Using the estimate \eqref{eq0.01} from Lemma \ref{ele-1}, we get
\begin{equation}\label{eq0.17}
g(t,\xi)\leq \left(\frac{|s|}{p|\Phi(\xi)|}\right)^{\frac{|s|}p}e^{-\frac{|s|}p}|\xi|^{|s|}.
\end{equation}

Since $M>1$ is large, an application of the estimate \eqref{eq0.6x34} from Lemma \ref{xav4} in \eqref{eq0.17}, yields
\begin{equation}\label{eq0.18}
g(t,\xi)\leq \left(\frac{2|s|}{p|\xi|^p}\right)^{\frac{|s|}p}e^{-\frac{|s|}p}|\xi|^{|s|} \leq \Big(\frac{2|s|}p\Big)^{\frac{|s|}p}e^{-\frac{|s|}p}.
\end{equation}

In light of the estimate \eqref{eq0.18}, one obtains that
\begin{equation}\label{eq0.181}
t^{\frac{|s|}p}\||\xi|^{|s|}e^{t\Phi(\xi)}\chi_{\{|\xi|> M\}}\|_{L^{\infty}}\leq \Big(\frac{2|s|}p\Big)^{\frac{|s|}p}e^{-\frac{|s|}p}.
\end{equation}

Inserting estimates \eqref{eq0.11}, \eqref{eq0.16} and \eqref{eq0.181} in \eqref{eq0.14},  we get
\begin{equation}\label{eq0.19}
t^{\frac{|s|}p}\|V(t)u_0\|_{L^2}\leq \Big(\frac{2|s|}p\Big)^{\frac{|s|}p}e^{-\frac{|s|}p}\|u_0\|_{H^s}\lesssim \|u_0\|_{H^s}.
\end{equation}

The conclusion of the Lemma follows from \eqref{eq0.12} and \eqref{eq0.19}.
\end{proof}

\begin{lemma}\label{lem2.4}
Let $0<T \leq 1$ and $t\in [0, T]$. Then for all $s\in \R$, we have
\begin{equation}\label{eq0.90}
\|V(t)u_0\|_{Y_T^s}\lesssim  e^{TC_M} \|u_0\|_{H^s},
\end{equation}
where the constant $C_M$ depends on $M$ with $M$ as in Lemma \ref{xav4}.
\end{lemma}

\begin{proof}
The estimate for the first component of the $Y_T^s$-norm has already been obtained in \eqref{eq0.12}. In what follows, we estimate the second  component of the $Y_T^s$-norm. We only consider the case when $s<0$. In the case when $s\geq 0$ the estimates follow easily. Using Plancherel identity, we have
\begin{equation}\label{eq0.130}
\begin{split}
t^{\frac{1+|s|}{p}}\|\partial_xV(t)u_0\|_{L^2} &= t^{\frac{1+|s|}{p}}\|\xi e^{t\Phi(\xi)}\widehat{u_0}\|_{L^2} \\
&= t^{\frac{1+|s|}{p}}\|\xi\langle\xi\rangle^{-s}e^{t\Phi(\xi)}\langle\xi\rangle^{s}\widehat{u_0}\|_{L^2}\\ 
&\leq t^{\frac{1+|s|}{p}}\|\xi\langle\xi\rangle^{|s|}e^{t\Phi(\xi)}\|_{L^{\infty}}\|u_0\|_{H^s}.
\end{split}
\end{equation}

Now,
\begin{equation}\label{eq0.140}
\begin{split}
t^{\frac{1+|s|}{p}}\|\xi\langle\xi\rangle^{|s|}e^{t\Phi(\xi)}\|_{L^{\infty}}
&\leq  t^{\frac{1+|s|}{p}}\|\xi\langle\xi\rangle^{|s|}e^{t\Phi(\xi)}\chi_{\{|\xi|\leq M\}}\|_{L^{\infty}}
+ t^{\frac{1+|s|}{p}}\|\xi\langle\xi\rangle^{|s|}e^{t\Phi(\xi)}\chi_{\{|\xi|> M\}}\|_{L^{\infty}}\\
&=:J_1+J_2.
\end{split}
\end{equation}

Since $\langle\xi\rangle^{|s|} \lesssim 1+|\xi|^{|s|}$,  and $t\in [0, T]$ with $0\leq T\leq 1$, we have
 \begin{equation}\label{eq0.110}
J_1\lesssim C_Mt^{\frac{1+|s|}{p}} \leq C_M.
\end{equation}

Now, we move to estimate the high-frequency part  $J_2$. For this, we use the estimate \eqref{eq0.01} from Lemma \ref{ele-1} with $b= \Phi(\xi) <0$ and $a =\frac{1+|s|}p$, to get
\begin{equation}\label{eq0.170}
e^{t\Phi(\xi)} \leq \left(\frac{ a e^{-1}}{|\Phi(\xi)|}\right)^{a}\frac1{t^a}.
\end{equation}

Since $M>1$ is large, $\langle\xi\rangle^{|s|} \lesssim |\xi|^{|s|}$, an application of the estimate \eqref{eq0.6x34} from Lemma \ref{xav4} in \eqref{eq0.170}, yields
\begin{equation}\label{eq0.180}
 t^{\frac{1+|s|}{p}}|\xi|^{1+|s|}\left(\frac{ a e^{-1}}{|\Phi(\xi)|}\right)^{a}\frac1{t^a} \lesssim t^{\frac{1+|s|}{p}-a}|\xi|^{1+|s|-ap} \lesssim C_M,
\end{equation}
and consequently
\begin{equation}\label{j-2}
J_2\lesssim C_M.
\end{equation}

The conclusion of the Lemma follows from \eqref{eq0.12}, \eqref{eq0.130}, \eqref{eq0.140}, \eqref{eq0.110}  and \eqref{j-2}.
\end{proof}

\begin{lemma}\label{xav6}
Let $-\frac p2<s$, $p>3$ and $\tau \in (0, 1]$. Then we have
\begin{equation}\label{eq0.9x1}
\|\xi \langle \xi \rangle^s e^{\tau \Phi(\xi)}\|_{L^2_\xi}\lesssim  \dfrac{1}{\tau^{\frac12+ \frac{s}p}},
\end{equation}
and
\begin{equation}\label{eq0.9x2}
\|\xi e^{\tau \Phi(\xi)}\|_{L^2_\xi}\lesssim  \dfrac{1}{\tau^{\frac{3^+}{2p}}}.
\end{equation}

\end{lemma}
\begin{proof}
In order to prove \eqref{eq0.9x1}, let $M$ be as in Lemma \ref{xav4}, and decompose the integral
\begin{equation}\label{eq0.17xavx1}
\begin{split}
\|\xi \langle \xi \rangle^s e^{\tau \Phi(\xi)}\|_{L^2_\xi}^2 & =\int_{|\xi| \leq M} \xi^2 \langle \xi \rangle^{2s} e^{2\tau \Phi(\xi)}d\xi +\int_{|\xi| \geq M} \xi^2 \langle \xi \rangle^{2s} e^{2\tau \Phi(\xi)}d\xi=:I_1+I_2.
\end{split}
\end{equation}
In the first integral, since $1+ 2\frac{s}p>0$ and $\tau \in (0,1]$ we have
\begin{equation}\label{eq0.17xavx2}
I_1 \le \int_{|\xi| \leq M} M^2 e^{2C \tau} d\xi \le 2M^3 e^{2C \tau} \le  \dfrac{2M^3 e^{2C}}{\tau^{1+ 2\frac{s}p}}.
\end{equation}

Now, we consider the second integral in \eqref{eq0.17xavx1}.  
For sufficiently large $M$, if we take $b=2 \Phi(\xi) <0$ (see Lemma \ref{xav4}) and $a=1+ 2\frac{s}p>0$, then using the estimates \eqref{eq0.01} and \eqref{eq0.6x34}, we get
\begin{equation*}
\begin{split}
I_2 &\lesssim \frac1{\tau^a} \int_{|\xi| \geq M} \xi^2 \langle \xi \rangle^{2s} \dfrac{1}{ |\Phi(\xi)|^a}\, d\xi \le \frac1{\tau^{1+ 2\frac{s}p}} \int_{|\xi| \geq M}\dfrac{1}{|\xi|^{-2-2s+p\,(1+ 2\frac{s}p)}} \, d\xi \lesssim \frac1{\tau^{1+ 2\frac{s}p}},
\end{split}
\end{equation*}
where in the last inequality  the fact that $-2-2s+p\,(1+ 2s/p)=p-2>1$ has been used, 
and this proves \eqref{eq0.9x1}.

The proof of the \eqref{eq0.9x2} is very similar. Again we consider $M$  as in Lemma \ref{xav4}, and decompose the integral
\begin{equation}\label{eq0.17xavx4}
\begin{split}
\|\xi  e^{\tau \Phi(\xi)}\|_{L^2_\xi}^2 & =\int_{|\xi| \leq M} \xi^2  e^{2\tau \Phi(\xi)}d\xi +\int_{|\xi| \geq M} \xi^2  e^{2\tau \Phi(\xi)}d\xi=:J_1+J_2.
\end{split}
\end{equation}
Since $\frac p3>0$ and $\tau \in (0,1]$, we have
\begin{equation}\label{eq0.17xavx5}
J_1 \le \int_{|\xi| \leq M} M^2 e^{2C \tau} d\xi \le 2M^3 e^{2C \tau} \le  \dfrac{2M^3 e^{2C}}{\tau^{\frac{3^+}p}}.
\end{equation}
Similarly as in the case of $I_2$, using \eqref{eq0.01} with $b=2 \Phi(\xi) <0$ and $a=\frac{3^+}p>0$, and estimate \eqref{eq0.6x34}, we obtain
\begin{equation*}
\begin{split}
J_2 &\lesssim \frac1{\tau^a} \int_{|\xi| \geq M} \xi^2  \dfrac{1}{ |\Phi(\xi)|^a}\, d\xi \le \frac1{\tau^{\frac{3^+}p}} \int_{|\xi| \geq M}\dfrac{1}{|\xi|^{-2+p\,(\frac{3^+}p)}} \, d\xi \lesssim \frac1{\tau^{\frac{3^+}p}},
\end{split}
\end{equation*}
where  in the last inequality the fact that $-2+p\,\big(\frac{3^+}p\big)>1,$
has been used, and this proves \eqref{eq0.9x2}.
\end{proof}

\begin{proposition}\label{xav7}
Let $-\frac p2<s\le 0$, $p>3$, $0<T \leq 1$ and $t\in [0, T]$. Then we have
\begin{equation}\label{eq0.9x3}
\left\|\int_0^t V(t-t')\partial_x(u v) (t') dt'\right\|_{X_T^s}\lesssim  T^{\alpha} \|u\|_{X_T^s}\|v\|_{X_T^s},
\end{equation}
where $\alpha= \frac{2s+p}{2p}>0$.

\end{proposition}

\begin{proof}

Using the definition of $V(t)$ and Minkowski's inequality, we have
\begin{equation}\label{proxav1}
\begin{split}
\left\|\int_0^t V(t-t')\partial_x(u v) (t')dt'\right\|_{H^s} & \le \int_0^t \| \xi \langle \xi \rangle^s e^{(t-t')\Phi(\xi)} (\widehat{u(t')}\ast \widehat{v(t')})dt'\|_{L_\xi^2}\\
& \le \int_0^t \| \xi \langle \xi \rangle^s e^{(t-t')\Phi(\xi)}\|_{L_\xi^2} \|(\widehat{u(t')}\ast \widehat{v(t')})(\xi)\|_{L_\xi^\infty}dt'.
\end{split}
\end{equation}
The Young's  inequality, Plancherel identity and definition of $X_T^s$ norm yield
\begin{equation}\label{convinf}
\|(\widehat{u(t')}\ast \widehat{v(t')})(\xi)\|_{L_\xi^\infty} \le t'^{-\frac{2|s|}p}\|u\|_{X_T^s}\|v\|_{X_T^s}.
\end{equation}
Combining inequalities \eqref{proxav1}, \eqref{convinf} and inequality \eqref{eq0.9x1} in Lemma \ref{xav6}, we get
 \begin{equation}\label{proxav2}
\left\|\int_0^t V(t-t')\partial_x(u v) (t')dt'\right\|_{H^s}
\lesssim \|u\|_{X_T^s}\|v\|_{X_T^s} \int_0^t \dfrac{1}{|t-t'|^{\frac12+\frac{s}p}\,|t'|^{\frac{2|s|}p} } dt'
\end{equation}
Making a change of variables $t'= t\tau$, we get
\begin{equation}\label{ea1.23}
\begin{split}
\left\|\int_0^t V(t-t')\partial_x(u v) (t')dt'\right\|_{H^s}
& \lesssim  t^{\frac{p+2s}{2p}}\, \|u\|_{X_T^s}\|v\|_{X_T^s}\int_0^1 \dfrac{1}{|1-\tau|^{\frac12+\frac{s}p}\,|\tau|^{\frac{2|s|}p} } d\tau\\
& \lesssim  t^{\frac{p+2s}{2p}}\, \|u\|_{X_T^s}\|v\|_{X_T^s}.
\end{split}
\end{equation}
Similarly inequality \eqref{eq0.9x2} in Lemma \ref{xav6} and \eqref{convinf} give
\begin{equation}\label{proxav3}
 t^{\frac{|s|}p}\left\|\int_0^t V(t-t')\partial_x(u v) (t')dt'\right\|_{L^2}
\lesssim \|u\|_{X_T^s}\|v\|_{X_T^s} t^{\frac{|s|}p} \int_0^t \dfrac{1}{|t-t'|^{\frac{3^+}{2p}}\,|t'|^{\frac{2|s|}p} } d\tau
\end{equation}

Again, Making a change of variables $t'= t\tau$, one has
\begin{equation}\label{eq1.24}
\begin{split}
t^{\frac{|s|}p}\left\|\int_0^t V(t-t')\partial_x(u v) (t')dt'\right\|_{L^2}&\lesssim  t^{\frac{2p+2s-3^+}{2p}}\|u\|_{X_T^s}\|v\|_{X_T^s} 
\int_0^1 \dfrac{1}{|1-\tau|^{\frac{3^+}{2p}}\,|\tau|^{\frac{2|s|}p} } d\tau\\
& \lesssim t^{\frac{2p+2s-3^+}{2p}} \|u\|_{X_T^s}\|v\|_{X_T^s}.
\end{split}
\end{equation}
\end{proof}

We also need the following estimate.
\begin{lemma}\label{xav6.6}
Let $1-\frac p2<s$, $p>3$ and $\tau \in (0, 1]$. Then we have
\begin{equation}\label{eq0.99}
\| \langle \xi \rangle^s e^{\tau \Phi(\xi)}\|_{L^2_\xi}\lesssim  \dfrac{1}{\tau^{\frac{p-2+2s}{2p}}}.
\end{equation}
\end{lemma}
\begin{proof}

For $M$ large as in Lemma \ref{xav4}, we have
\begin{equation}\label{eq2.03}
\|\langle \xi\rangle^s e^{\tau\Phi(\xi)}\|_{L^2}^2 = \int_{|\xi| \leq M}\langle \xi\rangle^{2s} e^{2\tau\Phi(\xi)}d\xi + \int_{|\xi| > M}\langle \xi\rangle^{2s} e^{2\tau\Phi(\xi)}d\xi =:A +B.
\end{equation}

Now, for $\tau\in (0, 1]$ and $a=\frac{p-2+2s}{p}>0$, one has 
\begin{equation}\label{eq2.04}
A\leq C_Me^{TC_M} \lesssim \frac1{\tau^a}.
\end{equation}

To obtain estimate for the high frequency part $B$, we use estimate \eqref{eq0.01} with $a= \frac{p-2+2s}{p}>0$ and $b= 2\Phi(\xi)<0$, to obtain
\begin{equation}\label{eq2.05}
B\leq \int_{|\xi| > M}\frac{|\xi|^{2s}}{\tau^a} \frac{(ae^{-1})^a}{|\Phi(\xi)|^a}d\xi \lesssim \int_{|\xi| > M}\frac1{|\xi|^{pa-2s}}\frac1{\tau^a} d\xi \lesssim \frac1{\tau^a},
\end{equation}
where in the last inequality $pa-2s >1$ has been used.
\end{proof}

\begin{proposition}\label{Prop2.7}
Let $1-\frac p2<s \le 0$, $p>3$, $0<T \leq 1$ and $t\in [0, T]$. Then we have
\begin{equation}\label{eq0.0}
\left\|\int_0^t V(t-t')(u_x v_x) (t') dt'\right\|_{Y_T^s}\lesssim  T^{\theta} \|u\|_{Y_T^s}\|v\|_{Y_T^s},
\end{equation}
where $\theta= \frac{p-2+2s}{2p}>0$.
\end{proposition}

\begin{proof}
We start considering the $H^s$ part of the $Y_T^s$-norm. Using the definition of $V(t)$ and Minkowski's inequality, we have
\begin{equation}\label{eq2.01}
\begin{split}
\left\|\int_0^t V(t-t')(u_x v_x) (t')dt'\right\|_{H^s} & \le \int_0^t \| \langle \xi \rangle^s e^{(t-t')\Phi(\xi)} (\widehat{u_x(t')}\ast \widehat{v_x(t')})dt'\|_{L_\xi^2}\\
& \le \int_0^t \| \langle \xi \rangle^s e^{(t-t')\Phi(\xi)}\|_{L_\xi^2} \|(\widehat{u_x(t')}\ast \widehat{v_x(t')})(\xi)\|_{L_\xi^\infty}dt'.
\end{split}
\end{equation}

The Young's  inequality, Plancherel identity and definition of $Y_T^s$ norm yield
\begin{equation}\label{eq2.02}
\|(\widehat{u_x(t')}\ast \widehat{v_x(t')})(\xi)\|_{L_\xi^\infty} \le t'^{-\frac{2(1+|s|)}{p}}\|u\|_{Y_T^s}\|v\|_{Y_T^s}.
\end{equation}

Using \eqref{eq0.99} and   \eqref{eq2.02} in \eqref{eq2.01}, we get
 \begin{equation}\label{eq2.06}
\left\|\int_0^t V(t-t')(u_x v_x) (t')dt'\right\|_{H^s}
\lesssim \|u\|_{Y_T^s}\|v\|_{Y_T^s} \int_0^t \dfrac{1}{|t-t'|^{\frac a2}\,|t'|^{\frac{2(1+|s|)}{p}} } dt'.
\end{equation}

Making a change of variables $t'= t\tau$, one obtains
\begin{equation}\label{eq2.07}
\left\|\int_0^t V(t-t')(u_x v_x) (t')dt'\right\|_{H^s}
 \lesssim  t^{1-\frac a2-\frac{2(1+|s|)}{p}}\, \|u\|_{Y_T^s}\|v\|_{Y_T^s}\int_0^1 \dfrac{1}{|1-\tau|^{\frac a2}\,|\tau|^{\frac{2(1+|s|)}{p}} } d\tau.
\end{equation}

For our choice of $a= \frac{p-2+2s}{p}$ and  $1/2>s> 1-\frac p2$ the integral in the RHS of \eqref{eq2.07} is finite, so we deduce that
\begin{equation}\label{eq2.08}
\left\|\int_0^t V(t-t')(u_x v_x) (t')dt'\right\|_{H^s}
 \lesssim  t^{\frac{p-2+2s}{2p}}\, \|u\|_{Y_T^s}\|v\|_{Y_T^s}.
\end{equation}

Now, we move to estimate the second part of the $Y_T^s$-norm. 
\begin{equation}\label{eq2.09}
\begin{split}
 \left\|\int_0^t\partial_x V(t-t')(u_x v_x) (t')dt'\right\|_{L^2}&\leq \int_0^t \|\xi e^{(t-t')\Phi(\xi)}\widehat{u_x(t')}\ast \widehat{v_x(t')} \|_{L^2}dt'\\
&\leq \int_0^t \|\widehat{u_x(t')}\ast \widehat{v_x(t')} \|_{L^{\infty}}\|\xi e^{(t-t')\Phi(\xi)}\|_{L^2}dt'.
\end{split}
\end{equation}

We have that  $\|\widehat{u_x(t)}\ast \widehat{v_x(t)} \|_{L^{\infty}}\leq t^{-\frac{2(1+|s|)}{p}}\|u\|_{Y_T^s}\|v\|_{Y_T^s}$. Taking $a=\frac{3^+}{2p}$, from \eqref{eq0.9x2},  one gets $\|\xi e^{\tau\Phi(\xi)}\|_{L^2}\lesssim \frac1{\tau^a}$. So, from \eqref{eq2.09}, one can deduce
\begin{equation}\label{eq2.010}
 \left\|\int_0^t \partial_xV(t-t')(u_x v_x) (t')dt'\right\|_{L^2}
\lesssim \|u\|_{Y_T^s}\|v\|_{Y_T^s}  \int_0^t \dfrac{1}{|t-t'|^{a}\,|t'|^{\frac{2(1+|s|)}{p}} } dt'.
\end{equation}

Making a change of variables $t'= t\tau$, one obtains from \eqref{eq2.010}
\begin{equation}\label{eq2.00}
t^{\frac{(1+|s|)}{p}}\left\|\int_0^t \partial_xV(t-t')(u_x v_x) (t')dt'\right\|_{L^2}\lesssim  t^{1-\frac{(1+|s|)}{p}-a}\|u\|_{Y_T^s}\|v\|_{Y_T^s} 
\int_0^1 \dfrac{1}{|1-\tau|^{a}\,|\tau|^{\frac{2(1+|s|)}{p}} } d\tau.
\end{equation}

For our choice of $a= \frac{3^+}{2p}$ and $s>1-\frac p2$  the integral in the RHS of \eqref{eq2.00} is finite.  Therefore, from \eqref{eq2.00}, we obtain
\begin{equation}\label{eq2.001}
\begin{split}
t^{\frac{1+|s|}{p}}\left\|\int_0^t \partial_xV(t-t')(u_x v_x) (t')dt'\right\|_{L^2}&\lesssim  t^{\frac{2p+2s-5^+}{2p}}\|u\|_{Y_T^s}\|v\|_{Y_T^s}\\
& \lesssim  t^{\frac{p-2+2s}{2p}}\|u\|_{Y_T^s}\|v\|_{Y_T^s}.
\end{split}
\end{equation}

Combining \eqref{eq2.08} and \eqref{eq2.001} we get the required estimate \eqref{eq0.0}.
\end{proof}

The following results deal with gain of regularity of the nonlinear part.
\begin{proposition}\label{xav7smoth}
Let $-\frac p2<s$, $p> 3$, $0 \le \mu < \frac p2$. If
\begin{equation}\label{xav3smoth}
\|f\|_{\mathcal{Z}_T^s}:=\sup_{t\in(0,T]}\Big\{\|f(t)\|_{H^s}+t^{\frac{|s|}{p}}\|f(t)\|_{L^2}\Big\}< \infty, 
\end{equation}
then  the application
\begin{equation}\label{eq0.9x3smoot}
t\mapsto \mathcal{L}(f)(t):=\int_0^t V(t-t')\partial_x(f^2) (t') dt',\qquad 0\le t\le T\le 1,
\end{equation}
is continuous from $[0, T]$ to  $H^{s+\mu}$. 
\end{proposition}

\begin{proof}
We start by proving that $\mathcal{L}(f)(t) \in H^{s+\mu}(\R)$ for all $f$ such that $\|f\|_{ \mathcal{Z}_T^s}<\infty$.
We consider two different cases

\noindent
{\bf Case I, $s\ge 0$:} Let $0< t\le T\le 1$, since $\langle \xi \rangle^{s}\le \langle \xi -x\rangle^{s}\langle x \rangle^{s}$ and $f\in Z_T^s$ we have
\begin{equation}
\begin{split}
\|\mathcal{L}(f)(t)\|_{H^{s+\mu}} &=\|\langle \xi \rangle^{s+\mu}\int_0^{t}\left( e^{(t-t')\Phi(\xi)} \right) \,i \xi \widehat{f} * \widehat{f}(\xi,t')dt'\|_{L^2}\\
&  \le \int_0^{t} \|\langle \xi \rangle^{\mu}\,\xi\,\left( e^{(t-t')\Phi(\xi)} \right) \|_{L^2}  \sup_{t' \in (0,T)}\|f(t') \|_{H^s}^2 dt'\\
&\lesssim  \|f \|_{\mathcal{Z}_T^s}^2  \int_0^{t}  \dfrac{1}{|t-t'|^{\frac12 +\frac{\mu}p}}\, dt' <\infty.
\end{split}
\end{equation}
where the definition of  $\mathcal{Z}_T^s$-norm, Minkowski's inequality and  inequality \eqref{eq0.9x1} from Lemma \ref{xav6} are used.

\noindent
{\bf Case II, $s\le 0$:} Similarly as in the proof of Proposition \ref{xav7}, we obtain
\begin{equation}
\begin{split}
\|\mathcal{L}(f)(t)\|_{H^{s+\mu}} &=\|\langle \xi \rangle^{s+\mu}\int_0^{t}\left( e^{(t-t')\Phi(\xi)} \right) \,i \xi \widehat{f} * \widehat{f}(\xi,t')dt'\|_{L^2}\\
&  \le \int_0^{t} \|\langle \xi \rangle^{s+\mu}\,\xi\,\left( e^{(t-t')\Phi(\xi)} \right) \|_{L^2}  \dfrac{1}{|t'|^{\frac{2|s|}{p}}} \|f \|_{\mathcal{Z}_T^s}^2 dt'\\
&\lesssim  \|f \|_{\mathcal{Z}_T^s}^2  \int_0^{t}  \dfrac{1}{|t-t'|^{\frac12 +\frac{s+\mu}p}}\, dt' <\infty.
\end{split}
\end{equation}

Now we  move to prove the continuity.
Let $t_0\in [0,T]$, fixed  and let $f$ such that $\|f\|_{ Z_T^s} <\infty$, we will shows that
\begin{equation}
\lim_{t \to t_0}\|\mathcal{L}(f)(t)-\mathcal{L}(f)(t_0)\|_{H^{s+\mu}}=0
\end{equation}
We use \eqref{eq0.9x3smoot} and the additive property of the integral, to get for $t \in [0,T]$ that
\begin{equation}
\begin{split}
\|\mathcal{L}(f)(t)&-\mathcal{L}(f)(t_0)\|_{H^{s+\mu}}
=\|\int_0^{t_0} V(t_0-t')\partial_x(f^{2})(t')d\tau-\int_0^t V(t-t')\partial_x(f^{2})(t')dt' \|_{H^{s+\mu}}\\
\le  &
\|\int_0^{t}\left( V(t_0-t')-V(t-t') \right)\partial_x(f^{2})(t')dt'\|_{H^{s+\mu}}+\|\int_{t}^{t_0} V(t-t') \partial_x(f^{2})(t')dt'\|_{H^{s+\mu}}\\
=& I_1(t,t_0)+I_2(t,t_0).
\end{split}
\end{equation}
We consider the first term 
\begin{equation}
\begin{split}
I_1(t,t_0)= & \|\langle \xi \rangle^{s+\mu}\int_0^{t}\left( e^{(t_0-t')\Phi(\xi)}-e^{(t-t')\Phi(\xi)} \right) \,i \xi \widehat{f} * \widehat{f}(\xi,t')dt'\|_{L^2}.
\end{split}
\end{equation}
As $( e^{(t_0-t')\Phi(\xi)}-e^{(t-t') })\to 0$ if $t \to t_0$, using the Lebesgue's Dominated Convergence Theorem we have that 
$$
I_1(t,t_0) \to 0, \qquad \textrm{if}\quad t\to t_0.
$$
Analogously, as 
$$
\int_{0}^{t_0} \|V(t-t') \partial_x(f^{2})(t')dt'\|_{H^{s+\mu}} <\infty
$$
we also have
$$
I_2(t,t_0) \to 0, \qquad \textrm{if}\quad t\to t_0,
$$
and this completes the proof.
\end{proof}

The next result follows by using \eqref{eq0.99} from Lemma \ref{xav6.6} and the procedure applied in Proposition \ref{Prop2.7}.
\begin{proposition}\label{xav7sm}
Let $1-\frac p2<s$, $p> 3$, $0 \le \mu < \frac p2$. If
\begin{equation}\label{xav3sm}
\|f\|_{\tilde{\mathcal{Z}}_T^s}:=\sup_{t\in(0,T]}\Big\{\|f(t)\|_{H^s}+t^{\frac{1+|s|}{p}}\|\partial_xf(t)\|_{L^2}\Big\}< \infty, 
\end{equation}
then the application
\begin{equation}\label{eq0.9x3sm}
t\mapsto\mathcal{L}(f)(t):=\int_0^t V(t-t')(f_x)^2 (t') dt',\qquad 0\le t\le T\le 1,
\end{equation}
is continuous from  $[0, T]$ to  $H^{s+\mu}$. 
\end{proposition}

\section{Proof of the well-posedness result}\label{sec-3}

This section is devoted to provide proofs of the local well-posedness results stated in Theorems \ref{teorp-1} and \ref{teorp}.
\begin{proof}[Proof of Theorem \ref{teorp-1}]
We consider the IVP \eqref{eq:hs} in its equivalent integral form
\begin{equation}\label{int1} 
v(t)=V(t)v_{0}- \int_{0}^{t}V(t-t')(v^{2})_x(t')dt',
\end{equation}
where 
$V(t)$ is the semigroup associated with the linear part given by (\ref{gV}).

We define an application
\begin{equation}\label{int2}
  \Psi(v)(t)= V(t)v_0-\int_0^t V(t-t')(v^{2})_x(t')dt'.
\end{equation}

 For $-\frac p2\leq s \leq  0$, $r>0$ and $0<T\leq 1$, let us define a ball
\begin{align*}
  B_r^T= \{f\in X_T^s ; \,\,\|f\|_{X_T^s}\leq r \}.
\end{align*}
We will prove that there exists $r>0$ and $0<T\leq 1$ such that the application  $\Psi$ maps  $B_r^T $ into 
$B_r^T$ and is a contraction. Let $v\in B_r^T$.
By using Lemma \ref{xav5} and Proposition \ref{xav7}, we get
\begin{equation}\label{eq3.5}
  \|\Psi(v)\|_{X_T^s} \leq  c\|v_0\|_{H^s}+c\,T^{\alpha} \|v\|_{X_T^s}^2,
\end{equation}
where $\alpha= \frac{2s+p}{2p}>0$. 

 Now, using the definition of $B_r^T$, one obtains
 \begin{equation}\label{eq3.8}
  \|\Psi(v)\|_{X_T^s}\leq \frac{r}{4}+ cT^\alpha r^{2}\leq  \frac{r}{2},
\end{equation}
where we have chosen $r=4c\|v_0\|_{H^s}$ and $cT^\alpha r=1/4$.
Therefore, from \eqref{eq3.8} we see that the application $\Psi$ maps $B_r^T$ into itself. A
similar argument proves that $\Psi$ is a contraction.  Hence $\Psi$ has a fixed point $v$ which is a  solution of the IVP (\ref{eq:hs}) such that $v \in C([0,T], H^s(\R))$. The smoothness of the solution map $v_0\mapsto v$ is a consequence of the contraction mapping principle using Implicit Function Theorem (for details see \cite{kpv1:kpv1}).

For the regularity part, we have from Lemma \ref{lem-smooth} that the linear part is in $ C([0, \infty); H^s(\R))\cap C((0, \infty); H^{\infty}(\R))$. Proposition \ref{xav7smoth} shows that the nonlinear part is in $C((0, T]); H^{s+\mu}(\R))$, $\mu>0$. Combining these information, we have $v\in C([0, T]; H^s(\R))\cap C((0, T]; H^{s+\mu}(\R))$. Rest of the proof follows a standard argument, so we omit the details.
\end{proof}

\begin{proof}[Proof of Theorem \ref{teorp}]
The proof of this theorem is similar to the one presented for Theorem \ref{teorp-1}. Here, we will use the estimates from Lemma \ref{lem2.4} and Proposition \ref{Prop2.7}. So, we omit the details.
\end{proof}


\section{Ill-posedness result}\label{sec-5}

In this section we will use the ideas presented in \cite{MR} to prove the ill-posedness result stated in Theorem \ref{illposed-1} and \ref{illposed-2}. The idea is to prove that  there are no spaces $X_T^s$ and $Y_T^s$ that are continuously embedded em $C([0, T]; H^s(\R))$ on which a contraction mapping argument can be applied. We start with the following result.

\begin{proposition}\label{prop4.1}
Let $p\ge 2$, $s<-\frac p2$ and $T>0$. Then there does not exist  a space $X_T^s$ continuously embedded in $C([0, T]; H^s(\R))$ such that
\begin{equation}\label{eq4.01}
\|V(t)v_0\|_{X_T^s}\lesssim \|v_0\|_{H^s},
\end{equation}
\begin{equation}\label{eq4.02}
\|\int_0^tV(t-t')\partial_x(v(t'))^2 dt'\|_{X_T^s}\lesssim \|v\|_{X_T^s}^2.
\end{equation}

\end{proposition}
\begin{proof}
The proof follows a contradiction argument. If possible, suppose that there exists a space $X_T^s$ that is continuously embedded in $C([0, T]; H^s(\R))$ such that the estimates \eqref{eq4.01} and \eqref{eq4.02} hold true.  If we consider $v= V(t)v_0$, then from \eqref{eq4.01} and \eqref{eq4.02}, we get
\begin{equation}\label{eq4.03}
\|\int_0^tV(t-t')\partial_x[V(t')v_0]^2 dt'\|_{H^s}\lesssim \|v_0\|_{H^s}^2.
\end{equation}

The main idea to complete the proof  is to find an appropriate initial data $v_0$ for which the estimate \eqref{eq4.03} fails to hold whenever $s<-\frac p2$.

Let $N\gg 1$, $0< \gamma \ll 1$, $I_N:= [N, N+2\gamma]$ and  define an initial data via Fourier transform
\begin{equation}\label{eq4.2}
\widehat{v_0}(\xi):= N^{-s}\gamma^{-\frac12}\big[\chi_{\{I_N\}}(\xi) + \chi_{\{-I_N\}}(\xi)\big].
\end{equation}
A simple calculation shows that $\|v_0\|_{H^s}\sim 1$.

Now, we move to calculate the $H^s$ norm of $f(x,t)$, where
\begin{equation}\label{eq4.04}
f(x,t):=\int_0^tV(t-t')\partial_x[V(t')v_0]^2 dt'.
\end{equation}

Taking the Fourier transform in the space variable $x$, we get
\begin{equation}\label{eq4.05}
\begin{split}
\widehat{f(t)}(\xi)&=\int_0^t e^{i(t-t')\xi^3+(t-t')\Phi(\xi)}i\xi \left(\widehat{V(t')v_0} * \widehat{V(t')v_0}\right) (\xi) dt'\\
&= \int_0^t e^{i(t-t')\xi^3+(t-t')\Phi(\xi)}i\xi\int_{\R}\widehat{v_0}(\xi-\xi_1) \widehat{v_0}(\xi_1) e^{it'\xi_1^3+t'\Phi(\xi_1)+it'(\xi-\xi_1)^3+t'\Phi(\xi-\xi_1)}d\xi_1dt'\\
&=i\xi e^{it\xi^3+t\Phi(\xi)}\int_{\R}\widehat{v_0}(\xi-\xi_1) \widehat{v_0}(\xi_1)\int_0^t e^{it'[-\xi^3+\xi_1^3+(\xi-\xi_1)^3]+t'[\Phi(\xi_1)-\Phi(\xi)+\Phi(\xi-\xi_1)]}dt'd\xi_1.
\end{split}
\end{equation}

We have that
\begin{equation}\label{eq4.06}
\int_0^t e^{it'[3\xi\xi_1(\xi_1-\xi)]+t'[\Phi(\xi_1)-\Phi(\xi)+\Phi(\xi-\xi_1)]}dt' = 
\frac{e^{it3\xi\xi_1(\xi_1-\xi)+t[\Phi(\xi_1)-\Phi(\xi)+\Phi(\xi-\xi_1)]}-1}{\Phi(\xi_1)-\Phi(\xi)+\Phi(\xi-\xi_1)+i3\xi\xi_1(\xi_1-\xi)}.
\end{equation}

Now, inserting \eqref{eq4.06} in \eqref{eq4.05}, one obtains
\begin{equation}\label{eq4.07}
\widehat{f(t)}(\xi)
=i\xi e^{it\xi^3}\int_{\R}\widehat{v_0}(\xi-\xi_1) \widehat{v_0}(\xi_1)\frac{e^{it3\xi\xi_1(\xi_1-\xi)+t\Phi(\xi_1)+t\Phi(\xi-\xi_1)}-e^{t\Phi(\xi)}}{\Phi(\xi_1)-\Phi(\xi)+\Phi(\xi-\xi_1)+i3\xi\xi_1(\xi_1-\xi)}d\xi_1.
\end{equation}

Therefore,
\begin{equation}\label{eqil1}
\| f\|_{H^s}^2 \gtrsim \int_{-\gamma/2}^{\gamma/2}\langle \xi\rangle^{2s} \dfrac{\xi^2}{N^{4s} \gamma^2}
\left| \int_{K} \dfrac{e^{3it\xi \xi_1 (\xi_1-\xi) +t \Phi(\xi_1)+t \Phi(\xi-\xi_1)}-e^{t \Phi(\xi)}}{ \Phi(\xi_1)-\Phi(\xi)+\Phi(\xi-\xi_1)+3i\xi \xi_1 (\xi_1-\xi)}d\xi_1\right|^2 d\xi,
\end{equation}
where 
$$
K= \{ \xi_1; \quad\xi-\xi_1 \in I_N, \xi_1\in -I_N\}\cup \{ \xi_1;\quad \xi_1 \in I_N, \xi-\xi_1\in -I_N\}.
$$
We have that $|K|\geq \gamma$ and
\begin{equation}\label{eqil2}
|3\xi \xi_1 (\xi_1-\xi)| \approx N^2 \gamma.
\end{equation}

In order to estimate \eqref{eqil1} we consider two cases:
\\
{ \bf Case  1:\,\, $\xi-\xi_1 \in I_N, \xi_1\in -I_N.$} In this case
\begin{equation}\label{eqil3}
\begin{split}
 &|\Phi(\xi_1)-\Phi(\xi)+\Phi(\xi-\xi_1)|  = |-(-\xi_1)^p+|\xi|^p-(\xi-\xi_1)^p+ \Phi_1(\xi_1)-\Phi_1(\xi)+\Phi_1(\xi-\xi_1)|\\
 & \leq |-2(-1)^p\xi_1^p|+|(\xi^p-p\xi^{p-1}\xi_1+ \cdots +p(-1)^{p-1}\xi \xi_1^{p-1})+|\xi|^p+ \Phi_1(\xi_1)-\Phi_1(\xi)+\Phi_1(\xi-\xi_1)|.
 \end{split}
\end{equation}
Therefore, 
\begin{equation}\label{eqil4}
\begin{split}
 |\Phi(\xi_1)-\Phi(\xi)+\Phi(\xi-\xi_1)|
 & \leq C (N^p+ N^{r}) \le 2C N^p, \quad r<p.
 \end{split}
 \end{equation}

Similarly we obtain 
$|\Phi(\xi_1)-\Phi(\xi)+\Phi(\xi-\xi_1)|
  \geq C (N^p- N^{r}) \gtrsim N^p$,  $r<p$. 
	
Hence
\begin{equation}\label{xaveqil4}
 |\Phi(\xi_1)-\Phi(\xi)+\Phi(\xi-\xi_1)|
 \sim   N^p.
 \end{equation}
 
 \noindent

{ \bf Case 2:\,\, $\xi_1\in I_N, \xi-\xi_1 \in- I_N$.} In this case
\begin{equation}\label{eqil5}
\begin{split}
 &|\Phi(\xi_1)-\Phi(\xi)+\Phi(\xi-\xi_1)|  = |-\xi_1^p+|\xi|^p-(-1)^p (\xi-\xi_1)^p+ \Phi_1(\xi_1)-\Phi_1(\xi)+\Phi_1(\xi-\xi_1)|\\
 & \leq |-2\xi_1^p|+ |(-1)^p (\xi^p-p\xi^{p-1}\xi_1+ \cdots +p(-1)^{p-1}\xi \xi_1^{p-1})+ |\xi|^p+\Phi_1(\xi_1)-\Phi_1(\xi)+\Phi_1(\xi-\xi_1)|.
 \end{split}
\end{equation}
In this way,
\begin{equation}\label{eqil6}
\begin{split}
 |\Phi(\xi_1)-\Phi(\xi)+\Phi(\xi-\xi_1)|
 & \leq C (N^p+ N^{r}) \le 2 C N^p, \quad r<p,
 \end{split}
\end{equation}
and analogously $|\Phi(\xi_1)-\Phi(\xi)+\Phi(\xi-\xi_1)| \gtrsim N^p$. 

Therefore,
 \begin{equation}\label{xaveqil6}
\begin{split}
 |\Phi(\xi_1)-\Phi(\xi)+\Phi(\xi-\xi_1)|
 \sim N^p.
 \end{split}
\end{equation}

Similarly for any  $\xi_1 \in K $, we get, for large $N$
\begin{equation}\label{eqil7}
\begin{split}
\Phi(\xi_1)+ \Phi(\xi-\xi_1) & = - |\xi_1|^p- |\xi-\xi_1|^p +\Phi_1(\xi_1)+ \Phi_1(\xi-\xi_1)\le - 2 N^p+ C N^{r}, \quad r<p\\
& \le -N^{p}.
\end{split}
\end{equation}
Let
$$
\dfrac{f}{g}:=\dfrac{e^{3it\xi \xi_1 (\xi_1-\xi) +t \Phi(\xi_1)+t \Phi(\xi-\xi_1)}-e^{t \Phi(\xi)}}{ \Phi(\xi_1)-\Phi(\xi)+\Phi(\xi-\xi_1)+3i\xi \xi_1 (\xi_1-\xi)},
$$
then
\begin{equation}\label{realfg}
\left|\textrm{Re}\left\{\dfrac{f}{g}\right\}\right|= \dfrac{|\textrm{Re}f \,\textrm{Re}g+ \textrm{Im}f\,\textrm{Im}g|}{|g|^2}\geq \dfrac{|\textrm{Re}f \,\textrm{Re}g|}{|g|^2}- \dfrac{ |\textrm{Im}f\,\textrm{Im}g|}{|g|^2}.
\end{equation}
For $\gamma \ll 1$, one can obtain
 \begin{equation}\label{eqil8}
 \begin{split}
\textrm{Re}f=\textrm{Re} \,\,\left\{ e^{3it\xi \xi_1 (\xi_1-\xi) +t \Phi(\xi_1)+t \Phi(\xi-\xi_1)}-e^{t \Phi(\xi)}\right\} 
&\le  e^{t \Phi(\xi_1)+t \Phi(\xi-\xi_1)}-e^{-t\gamma^p/2}\\
& \le e^{-tN^p}-e^{-t\gamma^p/2}\\
& \le \frac{-e^{-t\gamma^p/2}}{2},
\end{split}
\end{equation}
and also
\begin{equation}\label{xaveqil8}
 \begin{split}
\textrm{Im}f=\textrm{Im} \,\,\left\{ e^{3it\xi \xi_1 (\xi_1-\xi) +t \Phi(\xi_1)+t \Phi(\xi-\xi_1)}-e^{t \Phi(\xi)}\right\} 
&\le  e^{t \Phi(\xi_1)+t \Phi(\xi-\xi_1)}\\
& \le e^{-tN^p}.
\end{split}
\end{equation}
 From \eqref{eqil2}, \eqref{xaveqil4} and \eqref{xaveqil6} we conclude that for any $\xi_1 \in K $, one has
 \begin{equation}\label{eqil9}
 |\textrm{Re}g|=|\Phi(\xi_1)-\Phi(\xi)+\Phi(\xi-\xi_1)| \sim N^p,\quad |\textrm{Im}g|=3|\xi \xi_1 (\xi_1-\xi)| \sim N^2\gamma.
 \end{equation}
Using \eqref{eqil9}, \eqref{eqil8}, \eqref{xaveqil8} and \eqref{realfg} considering $N$ very large, it follows that
\begin{equation}\label{realfg-1}
\begin{split}
\left|\textrm{Re}\left\{\dfrac{f}{g}\right\}\right| \gtrsim \dfrac{-e^{-t\gamma^p/2} N^p}{N^{2p}+\gamma^2N^4}-\dfrac{e^{-tN^p}}{N^{p}+\gamma N^2} \gtrsim \dfrac{-e^{-t\gamma^p/2} N^p}{N^{2p}+\gamma^2N^4}
\end{split}
\end{equation}
Combining \eqref{eqil1}, \eqref{eqil9} and \eqref{eqil8}, using that $|z| \ge -\textrm{Re}z$, we arrive
\begin{equation}
\|f\|_{H^s}^2 \gtrsim  \gamma^{-2}N^{-4s}\gamma \langle \gamma \rangle^{2s}\gamma^2\dfrac{e^{-t\gamma^p}N^{2p}}{(N^{2p}+\gamma^2 N^4)^2}\gamma.
\end{equation}
Taking $\gamma \sim 1$ and $N$ very large, we obtain
\begin{equation}
\|f\|_{H^s}^2 \gtrsim \left\{
\begin{aligned} N^{-4s-2p}, \quad \textrm{if $p \ge 2$},\\
N^{-4s+2p-8}, \quad \textrm{if $p \le 2$},
\end{aligned}
\right.
\end{equation}
and this is a contradiction if $-4s-2p >0$ for $p \ge 2$ and  if $-4s+2p-8>0$ for $p\le 2$ or equivalently $s<-\frac p2$ for $p \ge 2$ and $s<\frac p2 -2$ for $0\le p \le 2$.
\end{proof}

\begin{proof}[Proof of Theorem \ref{illposed-1}]
For $v_0\in H^s(\R)$, consider the Cauchy problem
\begin{equation}\label{ivp3}
\begin{cases}
v_t+v_{xxx}+\eta Lv+(v^{2})_x=0, \quad x \in \mathbb{R}, \; t\geq 0,\\
     v(x,0)=\epsilon v_0(x),
\end{cases}
\end{equation}
where $\epsilon >0$ is a parameter. The solution $v^{\epsilon}(x,t)$ of \eqref{ivp3} depends on the parameter $\epsilon$. We can write \eqref{ivp3} in the equivalent integral equation form as
\begin{equation}\label{int-1}
v^{\epsilon}(t)=\epsilon V(t)v_{0}-\int_{0}^{t}V(t-t')(v^{2})_x(t')dt',
\end{equation}
where, $V(t)$ is the unitary group describing the solution of the linear part of the IVP \eqref{ivp3}.

Differentiating $v^{\epsilon}(x,t)$ in \eqref{int-1} with respect $\epsilon$ and evaluating at $\epsilon =0$ we get
\begin{equation}\label{int-2}
\frac{\partial v^{\epsilon}(x,t)}{\partial \epsilon}\Big|_{\epsilon=0} = V(t)v_0(x) =:v_1(x)
\end{equation}
and
\begin{equation}\label{int-3}
\frac{\partial^2 v^{\epsilon}(x,t)}{\partial \epsilon^2}\Big|_{\epsilon=0} = 2\int_0^t V(t-t')\partial_x(v_1^2(x,t'))dt' =:v_2(x).
\end{equation}

If the flow-map is $C^2$ at the origin from $H^s(\R)$ to $C([-T, T];H^s(\R))$, we must have
\begin{equation}\label{eq-bilin}
\|v_2\|_{L_T^{\infty}H^s(\R)}\lesssim \|v_0\|_{H^s(\R)}^2.
\end{equation}

But from  Proposition \ref{prop4.1}  we have seen that the estimate \eqref{eq-bilin} fails to hold for $s<-\frac p2$ if we consider $v_0$ given by \eqref{eq4.2} and this completes the proof of the Theorem.
\end{proof}


Now, we move to prove an ill-posedness results to the IVP \eqref{eq:hs-1}
\begin{proposition}\label{prop4.1x}
Let $p\ge 2$, $s<1-\frac p2$ and $T>0$. Then there does not exist  a space $Y_T^s$ continuously embedded in $C([0, T]; H^s(\R))$ such that
\begin{equation}\label{eq4.01a}
\|V(t)u_0\|_{Y_T^s}\lesssim \|u_0\|_{H^s},
\end{equation}
\begin{equation}\label{eq4.02a}
\|\int_0^tV(t-t')(u_x(t'))^2 dt'\|_{Y_T^s}\lesssim \|u\|_{Y_T^s}^2.
\end{equation}
\end{proposition}
\begin{proof}
Analogously as in the proof of Proposition \ref{prop4.1} we consider the same $v_0$ as defined in \eqref{eq4.2}, we take $u_0:=v_0$ and we calculate the $H^s$ norm of $g(x,t)$, where
\begin{equation}\label{eq4.04x1}
g(x,t):=\int_0^tV(t-t')[\partial_x V(t')u_0]^2 dt'.
\end{equation}
We have
\begin{equation*}
\widehat{g(t)}(\xi)
=i\xi e^{it\xi^3}\int_{\R}(\xi-\xi_1)\widehat{v_0}(\xi-\xi_1) \xi_1\widehat{u_0}(\xi_1)\frac{e^{it3\xi\xi_1(\xi_1-\xi)+t\Phi(\xi_1)+t\Phi(\xi-\xi_1)}-e^{t\Phi(\xi)}}{\Phi(\xi_1)-\Phi(\xi)+\Phi(\xi-\xi_1)+i3\xi\xi_1(\xi_1-\xi)}d\xi_1.
\end{equation*}
and
\begin{equation}\label{eqil.1}
\| g\|_{H^s}^2 \gtrsim \int_{-\gamma/2}^{\gamma/2}\langle \xi\rangle^{2s} \dfrac{\xi^2}{N^{4s} \gamma^2}
\left| \int_{K_\epsilon} \dfrac{\xi_1(\xi-\xi_1)e^{3it\xi \xi_1 (\xi_1-\xi) +t \Phi(\xi_1)+t \Phi(\xi-\xi_1)}-e^{t \Phi(\xi)}}{ \Phi(\xi_1)-\Phi(\xi)+\Phi(\xi-\xi_1)+3i\xi \xi_1 (\xi_1-\xi)}d\xi_1\right|^2 d\xi,
\end{equation}
where 
$$
K_\epsilon= \{ \xi_1; \quad\xi-\xi_1 \in I_N, \xi_1\in -I_N\}\cup \{ \xi_1;\quad \xi_1 \in I_N, \xi-\xi_1\in -I_N\}.
$$
Same way as in the proof of Proposition \ref{prop4.1}, we obtain
\begin{equation}
\|g\|_{H^s}^2 \gtrsim  \gamma^{-2}N^{-4s}\gamma \langle \gamma \rangle^{2s}\gamma^2\dfrac{N^4e^{-t\gamma^p}}{(N^{p}+\gamma N^2)^2}\gamma.
\end{equation}

Taking $p\ge 2$, $\gamma \sim 1$ and $N$ very large, we obtain
$$
\|g\|_{H^s}^2 \gtrsim N^{-4s-2p+4},
$$
and this is a contradiction if $-4s-2p+4 >0$ or equivalently $s<1-\frac p2$.
\end{proof}

\begin{proof}[Proof of Theorem \ref{illposed-2}]
For $v_0\in H^s(\R)$, consider the Cauchy problem
\begin{equation}\label{ivp4x}
\begin{cases}
u_t+u_{xxx}+\eta Lu+(u_x)^{2}=0, \quad x \in \mathbb{R}, \; t\geq 0,\\
     u(x,0)=\epsilon u_0(x),
\end{cases}
\end{equation}
where $\epsilon >0$ is a parameter. The solution $u^{\epsilon}(x,t)$ of \eqref{ivp4x} depends on the parameter $\epsilon$. We can write \eqref{ivp4x} in the equivalent integral equation form as
\begin{equation}\label{int-1x1}
u^{\epsilon}(t)=\epsilon V(t)u_{0}- \int_{0}^{t}V(t-t')(u_x)^{2}(t')dt',
\end{equation}
where, $V(t)$ is the unitary group describing the solution of the linear part of the IVP \eqref{ivp4x}.

Differentiating $u^{\epsilon}(x,t)$ in \eqref{int-1x1} with respect $\epsilon$ and evaluating at $\epsilon =0$ we get
\begin{equation}\label{int-2x2}
\frac{\partial u^{\epsilon}(x,t)}{\partial \epsilon}\Big|_{\epsilon=0} = V(t)u_0(x) =:u_1(x)
\end{equation}
and
\begin{equation}\label{int-3x3}
\frac{\partial^2 u^{\epsilon}(x,t)}{\partial \epsilon^2}\Big|_{\epsilon=0} = 2\int_0^t V(t-t')(\partial_xu_1(x,t'))^2dt' =:u_2(x).
\end{equation}

If the flow-map is $C^2$ at the origin from $H^s(\R)$ to $C([-T, T];H^s(\R))$, we must have
\begin{equation}\label{eq-bilinx4}
\|u_2\|_{L_T^{\infty}H^s(\R)}\lesssim \|u_0\|_{H^s(\R)}^2.
\end{equation}

But from  Proposition \ref{prop4.1x} we have that the estimate \eqref{eq-bilinx4} fails to hold for $s<1-\frac p2$ if we consider $u_0:=v_0$ given by \eqref{eq4.2} and this completes the proof of the Theorem.
\end{proof}


\end{document}